\documentclass{amsart}

\usepackage[utf8x]{inputenc}
\usepackage{graphicx, color}
\usepackage{amsmath,amssymb,amsthm}
\usepackage{hyperref}

\newtheorem{thm}{Theorem}[section]
\newtheorem{cor}[thm]{Corollary}
\newtheorem{obs}[thm]{Observation}
\newtheorem{prop}[thm]{Proposition}
\newtheorem{lem}[thm]{Lemma}

\newtheorem{conj}[thm]{Conjecture}

\newcommand{\mc}[1]{\mathcal{#1}}

\begin{document}

\thispagestyle{empty}

\centerline{\Large\bf A Tutte polynomial inequality for lattice path matroids}

\vspace{10mm}

\centerline{Kolja Knauer$^{\small 1}$, Leonardo Mart\'inez-Sandoval$^{\small 2,\small 3}$, Jorge Luis Ram\'irez Alfons\'in$^{\small 3}$}

\medskip
\begin{small}

\centerline{$^{1}$Laboratoire d'Informatique Fondamentale, Aix-Marseille Universit\'e and CNRS,}
\centerline{Facult\'e des Sciences de Luminy, F-13288 Marseille Cedex 9, France}

\centerline{\texttt{kolja.knauer@lif.univ-mrs.fr}}

\medskip

\centerline{$^{2}$Instituto de Matem\'aticas, Universidad Nacional Aut\'onoma de M\'exico at Juriquilla}
\centerline{Quer\'etaro 76230, M\'exico}
\centerline{\texttt{leomtz@im.unam.mx}}

\medskip

\centerline{$^{3}$Institut Montpelliérain Alexander Grothendieck, Universit\'e de Montpellier}
\centerline{Place Eugéne Bataillon, 34095 Montpellier Cedex, France}
\centerline{\texttt{jramirez@um2.fr}}

\end{small}

 \bigskip\bigskip\noindent
 {\bf Abstract.}
Let $M$ be a matroid without loops or coloops and let $T(M;x,y)$ be its Tutte polynomial. In 1999 Merino and Welsh conjectured that $$\max(T(M;2,0), T(M;0,2))\geq T(M;1,1)$$ holds for graphic matroids. Ten years later, Conde and Merino proposed a multiplicative version of the conjecture which implies the original one. In this paper we prove the multiplicative conjecture for the family of lattice path matroids (generalizing earlier results on uniform and Catalan matroids). In order to do this, we introduce and study particular lattice path matroids, called {\em snakes}, used as building bricks to indeed establish a strengthening of the multiplicative conjecture as well as a complete characterization of the cases in which equality holds.





\section{Introduction}

An {\em orientation} of a graph $G$ is an assignment of a direction to each edge. An orientation of $G$ is said to be {\em acyclic} if it has no directed cycles and {\em totally cyclic} if each edge belongs to a directed cycle.
Let $\tau(G)$ be the number of spanning trees of $G$. Let $\alpha(G)$ be the number of acyclic orientations of $G$ and $\alpha^\ast(G)$ the number of totally cyclic orientations of $G$. The following conjectures have been raised by Conde and Merino~\cite{Conde2009} and Merino and Welsh~\cite{Merino1999}:

\begin{conj}[Graphic Merino-Welsh conjectures]

\label{conj:GMW}

For any graph $G$ with no bridges and no loops we have:

	\begin{enumerate}
	  \item $\max\left(\alpha(G),\alpha^\ast(G)\right)\geq \tau(G)$.
	  \item  $\alpha(G)+\alpha^\ast(G)\geq 2\cdot \tau(G)$. \hfill (Additive)
	  \item  $\alpha(G)\cdot\alpha^\ast(G)\geq\tau(G)^2$. \hfill (Multiplicative)
	\end{enumerate}
  
\end{conj}

 Conjecture~\ref{conj:GMW}.3 is the strongest version. It is easy to verify that it implies Conjecture~\ref{conj:GMW}.2, which in turn implies Conjecture~\ref{conj:GMW}.1. Nevertheless, the multiplicative version turns out to be the most manageable. There are partial results concerning these conjectures. For a graph $G$ on $n$ vertices Thomassen~\cite{Thomassen2010} showed that $\tau(G)\le \alpha(G)$ if $G$ has at most $\frac{16}{15}n$ edges or $G$ has maximum degree at most $3$ and $\tau(G)\le \alpha^*(G)$ if $G$ has at least $4n$ edges or $G$ is a planar triangulation. Thus, establishing Conjecture~\ref{conj:GMW}.1 in these cases. Ch\'avez-Lomel\'i et al.~\cite{ChavezL2011} proved Conjecture~\ref{conj:GMW}.1 for several families of graphs, including wheels and complete graphs. Noble and Royle~\cite{Nob-14} established Conjecture~\ref{conj:GMW}.3 for the class of series-parallel graphs. 

As noticed in~\cite{Conde2009} and~\cite{Merino1999}, Conjecture~\ref{conj:GMW} can be stated in terms of the Tutte polynomial $T(G;x,y)$ of $G$ since

\[
    \tau(G)=T(G;1,1),\quad \alpha(G)=T(G;2,0) \text{ and } \alpha^\ast(G)=T(G;0,2).
\]

We thus have the following natural  generalization to matroids.

\begin{conj}[Matroidal Merino-Welsh conjectures]

\label{conj:MMW}

Let $M$ be a matroid without loops or coloops and let $T(M;x,y)$ be its Tutte polynomial. Then:

	\begin{enumerate}
	  \item $\max\left(T(M;2,0),T(M;0,2)\right)\geq T(M;1,1)$.
	  \item $T(M;2,0)+T(M;0,2)\geq 2\cdot T(M;1,1)$. \hfill (Additive) 
	  \item $T(M;2,0)\cdot T(M;0,2)\geq T(M;1,1)^2$. \hfill (Multiplicative)
	\end{enumerate}
  
\end{conj}

Notice that not allowing loops and coloops in $M$ is a fundamental hypothesis for the multiplicative version since a loop would imply $T(M;2,0)=0$ and a coloop would imply $T(M;0,2)=0$. 
\smallskip

An important result related to the multiplicative Merino-Welsh conjecture due to Jackson~\cite{Jackson2010} is that $T(M;b,0) \cdot T(M;0,b)\ge T(M;a,a)^2$ for any loopless, coloopless matroid $M$ provided that $b\ge a(a+2)$.
Conjecture~\ref{conj:MMW}.1 for {\em paving} matroids, {\em Catalan} matroids, and whirls is proved in~\cite{ChavezL2011}.
By combining the results from~\cite{ChavezL2011} and~\cite{Jackson2010} it can be proved inductively that paving matroids even satisfy Conjecture~\ref{conj:MMW}.3 (the base cases need a detailed treatment).
\smallskip

The main contribution of the present paper is to prove Conjecture~\ref{conj:MMW}.3 for the class of {\em lattice path} matroids (which contains, in particular, the families of {\em Catalan} matroids and {\em uniform} matroids).

\begin{thm}
    \label{thm:MWLPM}
    Let $M$ be a lattice path matroid without loops or coloops that is not a direct sum of trivial snakes. Then,
    
    \[
      T(M;2,0)\cdot T(M;0,2)\geq \frac{4}{3} \cdot T(M;1,1)^2.
    \]
    
\end{thm}

Here, a {\em trivial snake} is a pair of parallel elements. We will see that direct sums of trivial snakes do not satisfy the inequality with the factor $\frac{4}{3}$ but they do satisfy Conjecture~\ref{conj:MMW}.3.

Our theorem is an improvement by a multiplicative constant, and thus it directly implies the multiplicative version of Conjecture~\ref{conj:MMW}. Furthermore, it enables us to characterize the lattice path matroids in which Conjecture~\ref{conj:MMW}.3. holds with equality as precisely being the direct sums of trivial snakes (Corollary~\ref{cor:MWLPM}).

In Section~\ref{sec:basic} we state some basic definitions and properties in matroid theory needed for the rest of the paper. Afterwards, in Section~\ref{sec:lpmsnakes}, we introduce lattice path matroids. We define snakes (which are matroids that can be thought of as ``thin'' lattice path matroids) and prove that they are graphic matroids. We provide explicit formulas for the number of bases and acyclic orientations snakes, which are crucial ingredients for the proof of our main result (Theorem~\ref{thm:MWLPM}), that will be given in Section~\ref{sec:mwconj}.


\section{Basic definitions and properties}
\label{sec:basic}


There are several ways to define a matroid. We refer to~\cite{Welsh1976} for a thorough introduction into the topic. In this paper, we will define matroids in terms of their bases. A \emph{matroid} is a pair $M=(E,\mathcal{B})$ consisting of a finite ground set $E$ and a collection $\mc{B}$ of subsets of $E$ which satisfies:

\begin{itemize}
	\item[(B0)] $\mathcal{B}$ is non-empty.
	\item[(B1)] If $A$ and $B$ are in $\mc{B}$ and there is an element $a\in A\setminus B$, then there exists an element $b\in B\setminus A$ such that $A\setminus \{a\}\cup \{b\}$ is in $\mc{B}$.
\end{itemize}

The elements of $\mc{B}$ are called \emph{bases}. 
If $E=\emptyset$, then $M$ is called \emph{empty}.

It is a basic fact in matroid theory that all the bases of a matroid $M$ have the same cardinality. This number is the \textit{rank} of the matroid. If an element $a\in E$ belongs to no base, it is called a \emph{loop}. If $a$ belongs to every base, it is called a \emph{coloop}. If a matroid has no loops and no coloops we will call it \textit{loopless-coloopless}, which we will abbreviate by \textit{LC}.

If $M$ is a matroid with base set $\mc{B}$ and ground set $E$ then one constructs another matroid $M^\ast$ called the \emph{dual of $M$} with the same ground set but with base set
\[
	\mc{B}^\ast:=\{E\setminus B: B\in \mc{B}\}.
\]

%
%
%
%
	
If $M$ and $N$ are matroids with disjoint ground sets $E$ and $F$, respectively, then the \textit{direct sum} of $M$ and $N$ is the matroid whose ground set is the union of $E$ and $F$, and whose bases are those sets which can be written as the union of a base of $M$ and a base of $N$. If a matroid cannot be expressed as the direct sum of two non-empty matroids it is said to be \emph{connected}, otherwise it is \emph{disconnected}. Note that every connected matroid is \textit{LC}, except for matroids whose ground set consists of a single element.
	
Let $M$ be a matroid with base set $\mathcal{B}$ and ground set $E$ and let $S\subseteq E$. The inclusion-maximal sets of $\{B\setminus S: B\in \mc{B}\}$ are the bases of a new matroid $M\setminus S$ called the \emph{deletion} of $S$. The dual construction is the \textit{contraction} of $S$. It can be defined as the matroid $M/S:=(M^\ast \setminus S)^\ast$. If $S=\{s\}$, then we abbreviate the notations $M\setminus \{s\}$ and $M/\{s\}$ by $M\setminus s$ and $M/s$, respectively. The deletion allows us to extend the notion of rank to subsets of the ground set: For a subset $A\subseteq E$ we denote the \emph{rank of $A$} by $r(A)$ which is defined as the rank of $M\setminus(E\setminus A)$. Note that $r(E)$ is thus the rank of $M$.
	
	A very useful algebraic invariant for matroids is the \textit{Tutte polynomial}. Given a matroid $M$, this is a two-variable polynomial defined as follows:
	
	\[
	  T(M;x,y)=\sum_{A\subseteq E} (x-1)^{r(E)-r(A)}(y-1)^{|A|-r(A)}.
	\]
	
	The Tutte polynomial contains important information about the matroid. Most importantly in our context, $T(M;1,1)$ is the number of bases of $M$. In the case of orientable and in particular graphic matroids, $T(M;2,0)$ and $T(M;0,2)$ count the number of acyclic and totally cyclic orientations of the underlying graph, respectively. It is also well known that the Tutte polynomial satisfies the following recursive property~\cite{Welsh1999}:
$$T(M;x,y)=\left\{\begin{array}{ll}
T(M\setminus s;x,y) + T(M/s; x,y) & \text{ if  } s \text{ is neither a loop}\\
&\text{ \hspace*{.6cm} nor a coloop,}\\
xT(M\setminus s;x,y) & \text{ if } s\text{ is a coloop,}\\
yT(M/s;x,y) & \text{ if } s \text{ is a loop}.
	\end{array}\right.$$

We point out that when $s$ is a loop or a coloop, $M\setminus s=M/s$. 
Furthermore, we will make use of the facts that $T(M;x,y)=T(M^*;y,x)$ and that if $L$ is the direct sum of $M$ and $N$, then $T(L;x,y)=T(M;x,y)\cdot T(N;x,y)$.


\section{Lattice path matroids and snakes}
\label{sec:lpmsnakes}

In this section we address the class of lattice path matroids first introduced by Bonin, de Mier, and Noy~\cite{Bon-03}. Many different aspects of lattice path matroids have been studied:  excluded minor results~\cite{Bon-10}, algebraic geometry notions~\cite{Del-12,Sch-10,Sch-11}, complexity of computing the Tutte polynomial~\cite{Bon-07,Mor-13}, and results around the matroid base polytope~\cite{Cha-11}.
\smallskip

In order to define lattice path matroids, we first introduce transversal matroids.
Let ${\mathcal A}=(A_j : j\in J)$ be a set system, that is, a multiset of subsets of a finite set $E$. A {\em transversal} of $\mathcal A$ is a set $\{x_j : j\in J\}\subseteq E$ of $|J|$ distinct elements such that
$x_j\in A_j$ for all $j$. 
A fundamental result of Edmonds and Fulkerson~\cite{Edmonds-65} states that the set of transversals of a system ${\mathcal A}=(A_j : j\in J)$ constitutes the base set of a matroid on $E$. (Note that the set of transversals could be empty, but this does not happen for the matroids considered in the paper.) The collection $\mathcal A$ is a {\em presentation} of this matroid and any matroid that arises in such a way from a set system is a \emph{transversal matroid}. 
\smallskip

A path in the plane is a {\em lattice path}, if it starts at the origin and only does steps of the form $+(1,0)$ and $+(0,1)$, called {\em North} ($N$) and {\em East} ($E$), respectively. One way to encode a lattice path $P$ is therefore to simply identify it with a sequence $P=(p_1,\dots ,p_{r+m})$, where $p_i\in\{N,E\}$ for all $1\leq i\leq r+m$. Let $\{p_{s_1},\dots ,p_{s_r}\}$ be the set of North steps of $P$ with $s_1<\cdots <s_r$. Clearly, if the total number of steps of $P$ is known, we can recover $P$ from $\{s_1, \ldots, s_{r}\}$, by setting $p_i=N$ if $i\in\{s_1, \ldots, s_{r}\}$ and $p_i=E$, otherwise, for all $1\leq i\leq r+m$. Let $P=\{s_1, \ldots, s_{r}\}$ and $Q=\{t_1,\dots ,t_r\}$ be two lattice paths encoded that way, both ending at the point $(m,r)$, such that $P$ never goes above $Q$. 
%
%
%
The latter condition is equivalent to $t_i\le s_i$ for all $1\le i\le r$.
The {\em lattice path matroid (LPM)} associated to $P$ and $Q$ is the transversal 
matroid $M[P,Q]$ on the ground set $\{1,\dots ,m+r\}$ and presentation
$(A_i : i\in \{1,\dots ,r\})$ where $A_i$ denotes the interval of integers between $t_i$ and $s_i$.
In~\cite[Theorem 3.3]{Bon-03} it was proved that a subset $B$ of $\{1,\dots ,m+r\}$ with $|B|=r$ is a base of $M[P,Q]$ if and only if the associated path, which we will also denote by $B$, stays in the region bounded by $P$ and $Q$, see Figure~\ref{fig:snake-ex}.  We call the part of the plane enclosed by $P$ and $Q$, the \emph{diagram of $M[P,Q]$}. 

\begin{figure}[ht] 
\centering
 \includegraphics[width=1\textwidth]{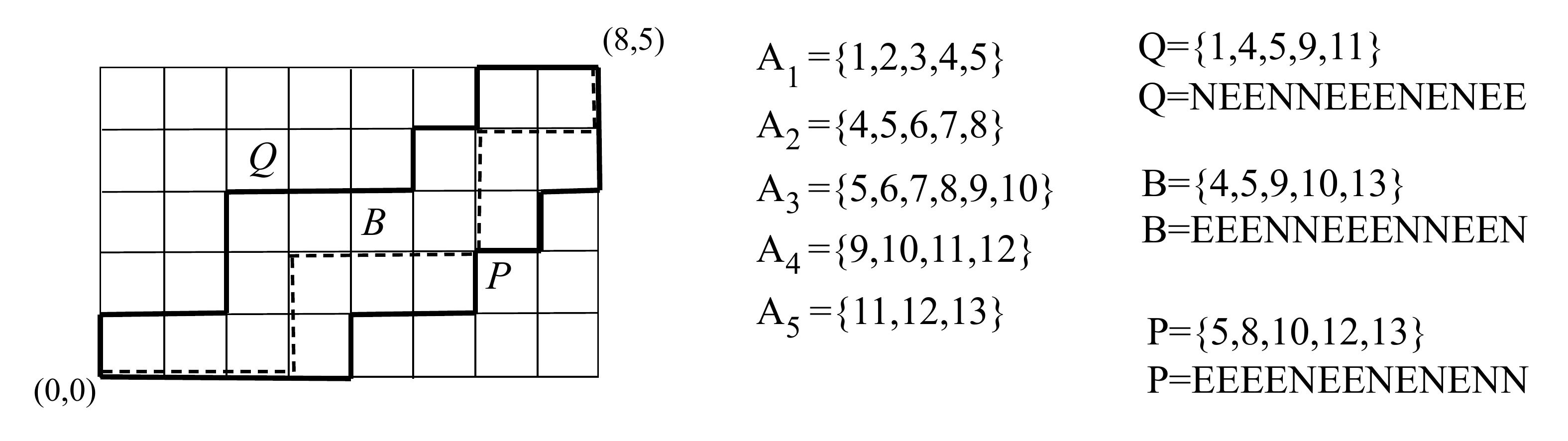}
 \caption{Left: Lattice paths $P$ and $Q$ from $(0, 0)$ to $(8, 5)$ and a path $B$ staying between $P$ and $Q$ in the diagram of $M[P,Q]$. Middle: The set system $A_1,\dots ,A_5$ representing $M[P,Q]$. Right: 
 Representations of $P$, $Q$, and $B$ as subsets of $\{1,\ldots, 13\}$ and as words in the alphabet $\{E,N\}$. }
 \label{fig:snake-ex}
\end{figure} 

The {\em uniform matroid} $U_{r,r+n}$ is the LPM $M[P,Q]$ with $Q =\{1, \dots, r\} = \underbrace{N\cdots N}_{r}\underbrace{E\cdots E}_{n}$ and
$P =\{n+1,n+ 2, \dots, n+r\} = \underbrace{E\cdots E}_{n}\underbrace{N\cdots N}_{r}$.
The {\em $k$-Catalan} matroid is the LPM $M[P,Q]$ with $Q =\{1, 3, \dots, 2k-1\} = \underbrace{NENE\cdots NE}_{k -pairs}$ and $P =\{k+1, k+2,\cdots , 2k\} =\underbrace{E\cdots E}_{k}\underbrace{N\cdots N}_{k}$.  
\smallskip

It is known~\cite[Theorem 3.4]{Bon-03} that the class of LPMs is closed under matroid duality. Indeed, for an LPM $M$, the bases of the dual matroid $M^*$ correspond to the East steps of the lattice paths in the diagram of $M$. Thus, reflecting the diagram of $M$ along the diagonal $x=y$, yields a diagram for $M^*$ and shows that $M^*$ is an LPM, as well. See Figure~\ref{fig:duallpm}. Furthermore, the class of LPMs is closed under duality~\cite{Bon-03}.

 \begin{figure}[ht] 
\centering
 \includegraphics[width=.5\textwidth]{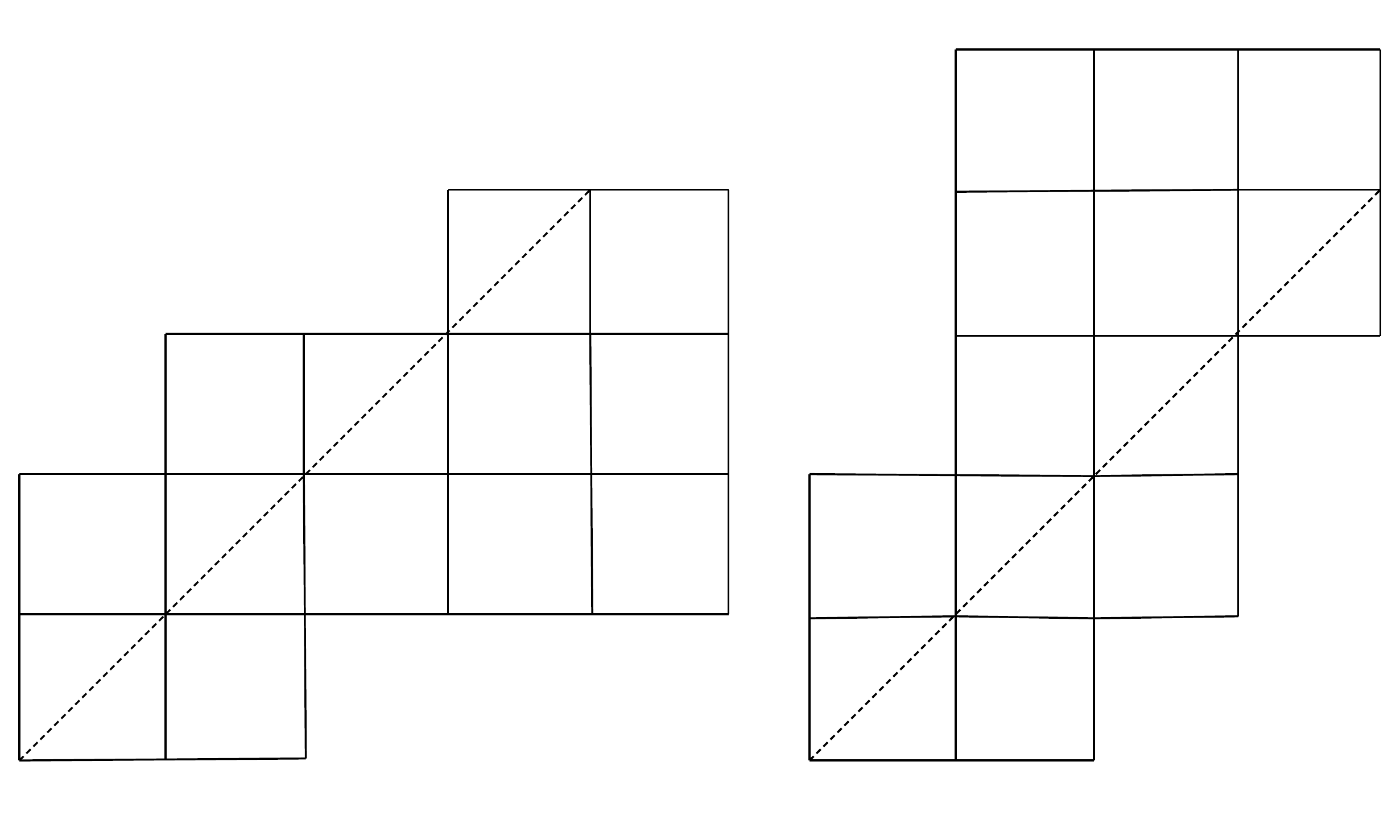}
 \caption{Presentations of an LPM and its dual.}
 \label{fig:duallpm}
\end{figure} 

The direct sum of LPMs in terms of their diagrams is illustrated in Figure~\ref{fig:directsumlpm}. In particular, we shall later use the fact (\cite[Theorem 3.6]{Bon-03}) that the LPM $M[P,Q]$ is connected if and only if the paths $P$ and $Q$ intersect only at $(0,0)$ and $(m,r)$. Moreover, we can detect loops and coloops in the diagram the following way. If $P$ and $Q$ share a horizontal (respectively vertical) edge at step $e$, then $e$ is a loop (respectively a coloop). Therefore, LC LPMs are those in which $P$ and $Q$ do not share vertical or horizontal edges.

 \begin{figure}[ht] 
\centering
 \includegraphics[width=.7\textwidth]{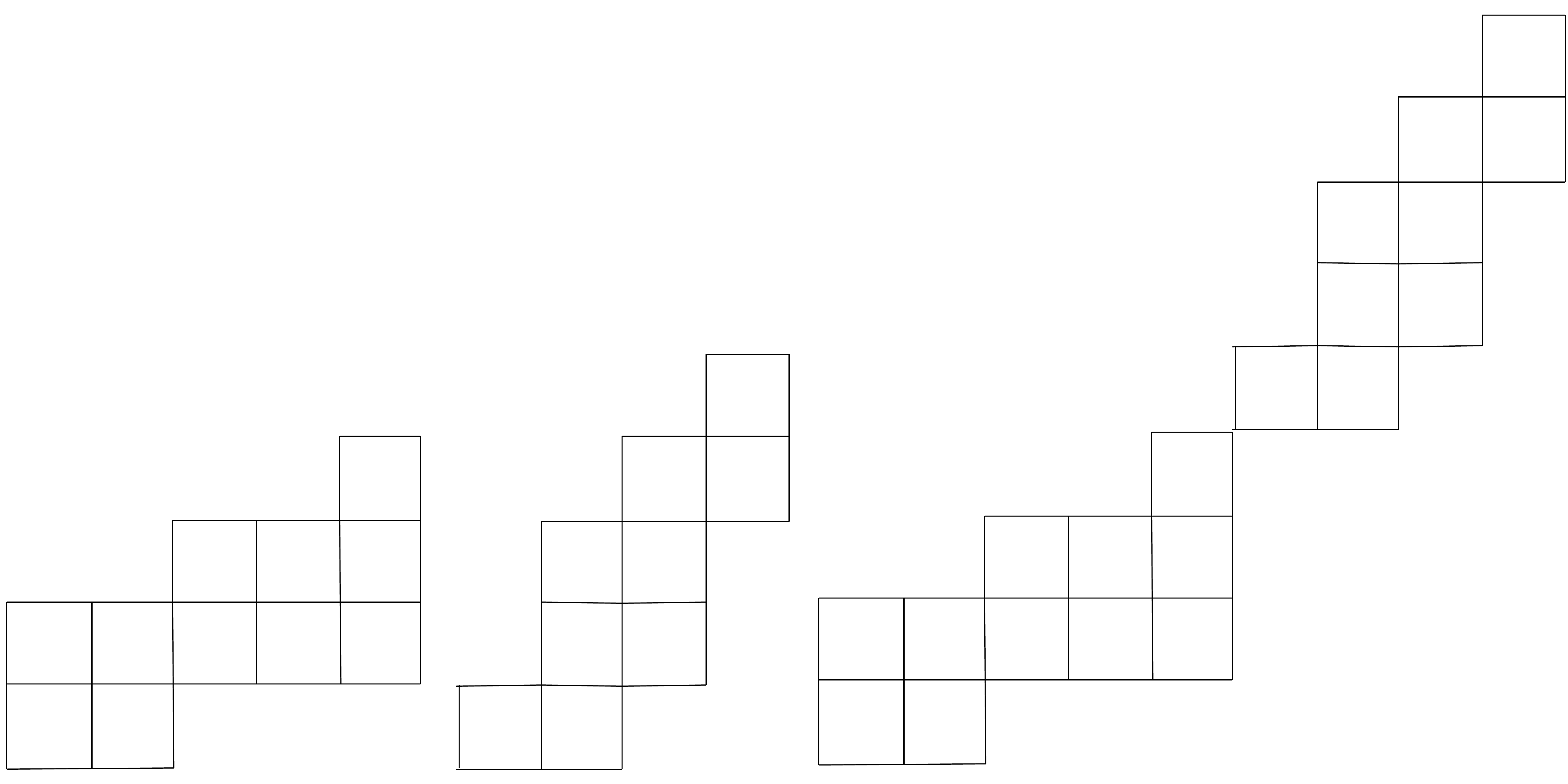}
 \caption{Diagrams of two LPMs and their direct sum.}
 \label{fig:directsumlpm}
\end{figure} 

\smallskip
	
In this paper we define a special class of LPMs, whose members are called snakes. An LPM is called \emph{snake} if it has at least two elements, is connected and has a diagram without interior lattice points. See Figure~\ref{fig:snakeslabels} for an example. 
We represent a snake as $S(a_1,a_2,\ldots,a_n)$ if starting from the origin its diagram encloses $a_1\ge 1$ squares to the right, then $a_2\ge 2$ squares up, then $a_3\ge 2$ squares to the right and so on up to $a_n\ge 2$, where the last square counted by each $a_i$ coincides with the first square counted by $a_{i+1}$ for all $i\leq n-1$. We call $S(1)$ the \emph{trivial snake} (one square).

\begin{figure}[ht] 
\centering
 \includegraphics[width=.5\textwidth]{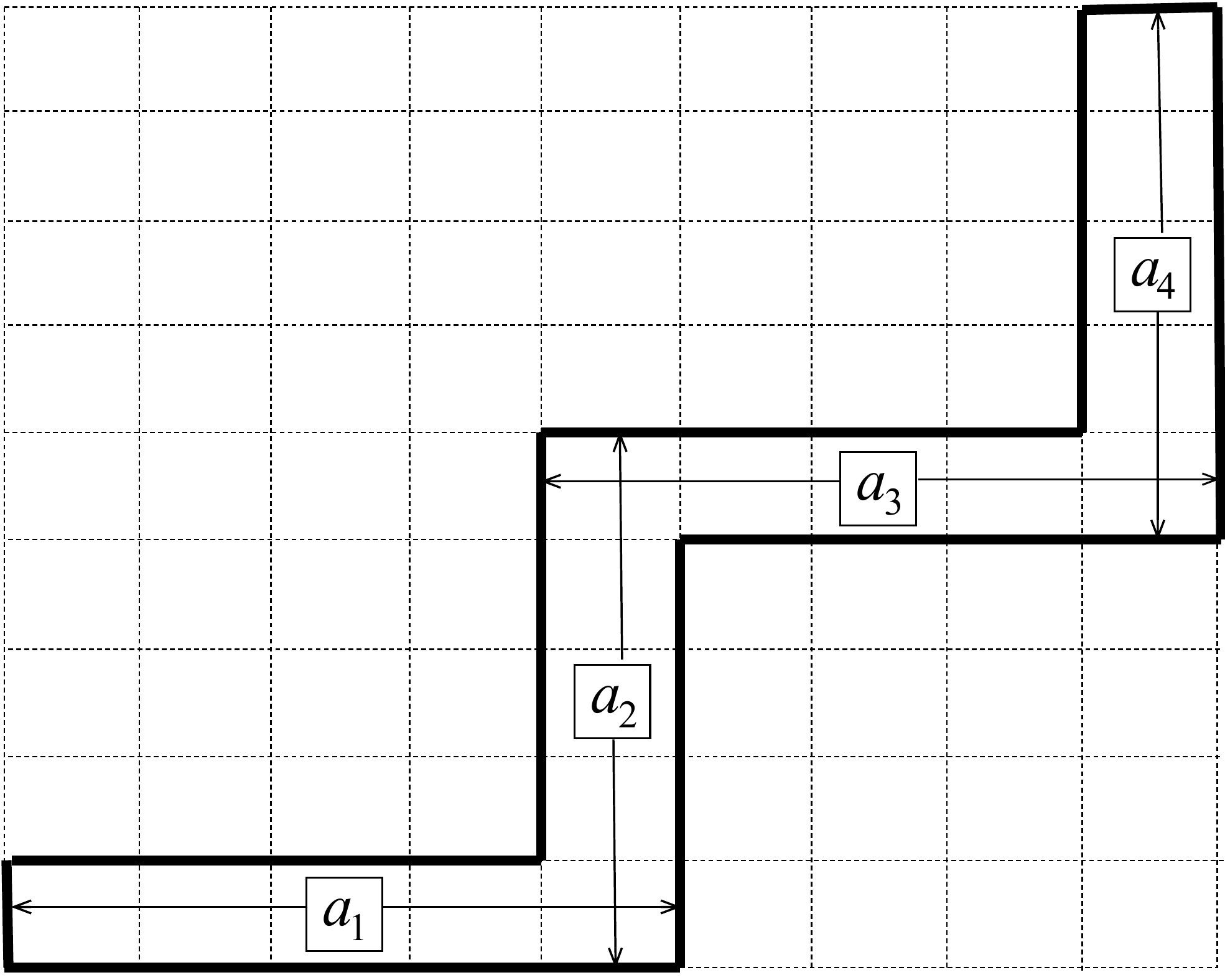}
 \caption{The diagram of a snake.}
 \label{fig:snakeslabels}
\end{figure}

The duality of LPMs depicted in Figure~\ref{fig:duallpm} restricts to snakes, i.e., the class of snakes is closed under duality. An example is illustrated in Figure~\ref{fig:snakeslabels1}.  More precisely, the following is easy to see:

\begin{obs}\label{obs:snakedual}
 Let $n$ be a positive integer and $a_1,\dots ,a_n$ be integers with $a_1\ge 1$ and $a_i\ge 2$ for all $2\leq i\leq n$. For the dual of $S(a_1,\ldots, a_n)$ we have $$S^*(a_1,\ldots, a_n)=\begin{cases} S(1,a_1,\ldots, a_n) &\mbox{if } a_1>1, \\
S(a_2,\ldots, a_n) & \mbox{if } a_1=1<n, \\
S(1) & \mbox{if } a_1=1=n. \end{cases}$$
\end{obs}

Observation~\ref{obs:snakedual} is useful since in some of our results it allows to assume $a_1\ge 2$ because (except for the trivial snake) the case $S(1,a_2,\dots ,a_n)$ can be treated via its dual $S^*(1,a_2,\dots ,a_n)=S(a_2,a_3,\dots ,a_n)$.

%


\begin{figure}[ht] 
\centering
 \includegraphics[width=.5\textwidth]{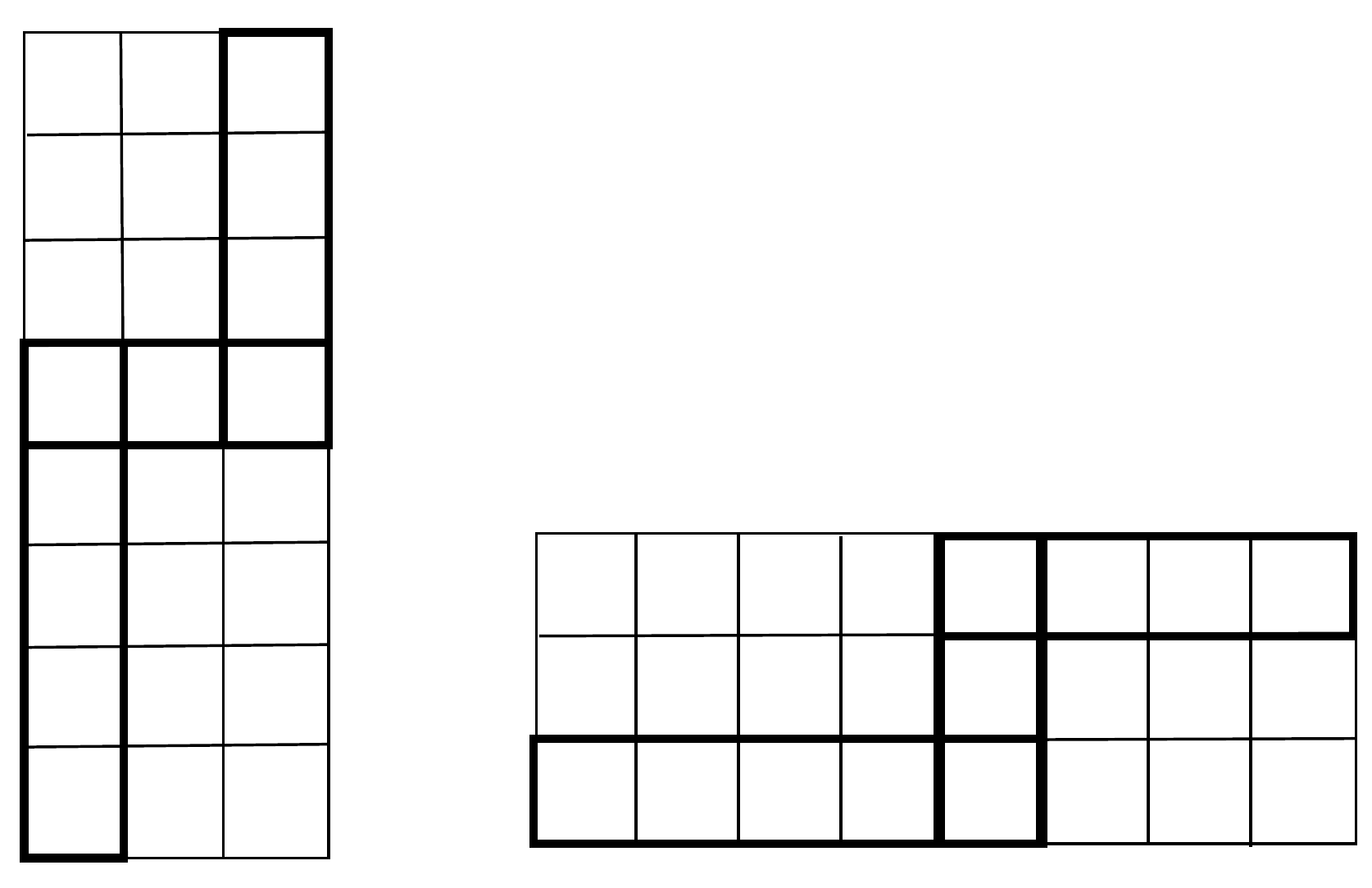}
 \caption{The snake $S(1,5,3,4)$ and its dual $S^*(1,5,3,4)=S(5,3,4)$.}
 \label{fig:snakeslabels1}
\end{figure}

The rest of this section is devoted to finding exact formulas for some values of the Tutte polynomial for snakes: $T(S;2,0)$, $T(S;0,2)$ and $T(S;1,1)$. These formulas will be crucial to prove our main result in Section~\ref{sec:mwconj}. For obtaining them, it will turn out to be useful to view snakes as graphic matroids.

%

To this end we introduce a special family of graphs, whose family of associated graphic matroids will turn out to coincide with the class of snakes.
Let $\ell$ be a positive integer and $ c =(c_1,c_2,\ldots, c_{\ell})$ and $ d=(d_1,d_2,\ldots,d_{\ell-1})$ be vectors of positive integers. The \emph{multi-fan} $F(c,d)$ is the graph consisting of a path $(v_1,v_1^2,\ldots, v_1^{d_1},v_2,v_2^2,\ldots, v_2^{d_2},\ldots,v_{\ell-1},v_{\ell-1}^2,\ldots, v_{\ell-1}^{d_{\ell-1}},v_{\ell})$ plus a single vertex $x$ that is connected by a bundle of $c_i\geq 1$ parallel edges to $v_i$ for every $1\leq i\leq \ell$, see Figure~\ref{fig:fanex}. If $\ell=1$, then the multi-fan $F(c_1)$ consists of a single bundle of $c_1$ parallel edges.

\begin{figure}[ht] 
\centering
 \includegraphics[width=.7\textwidth]{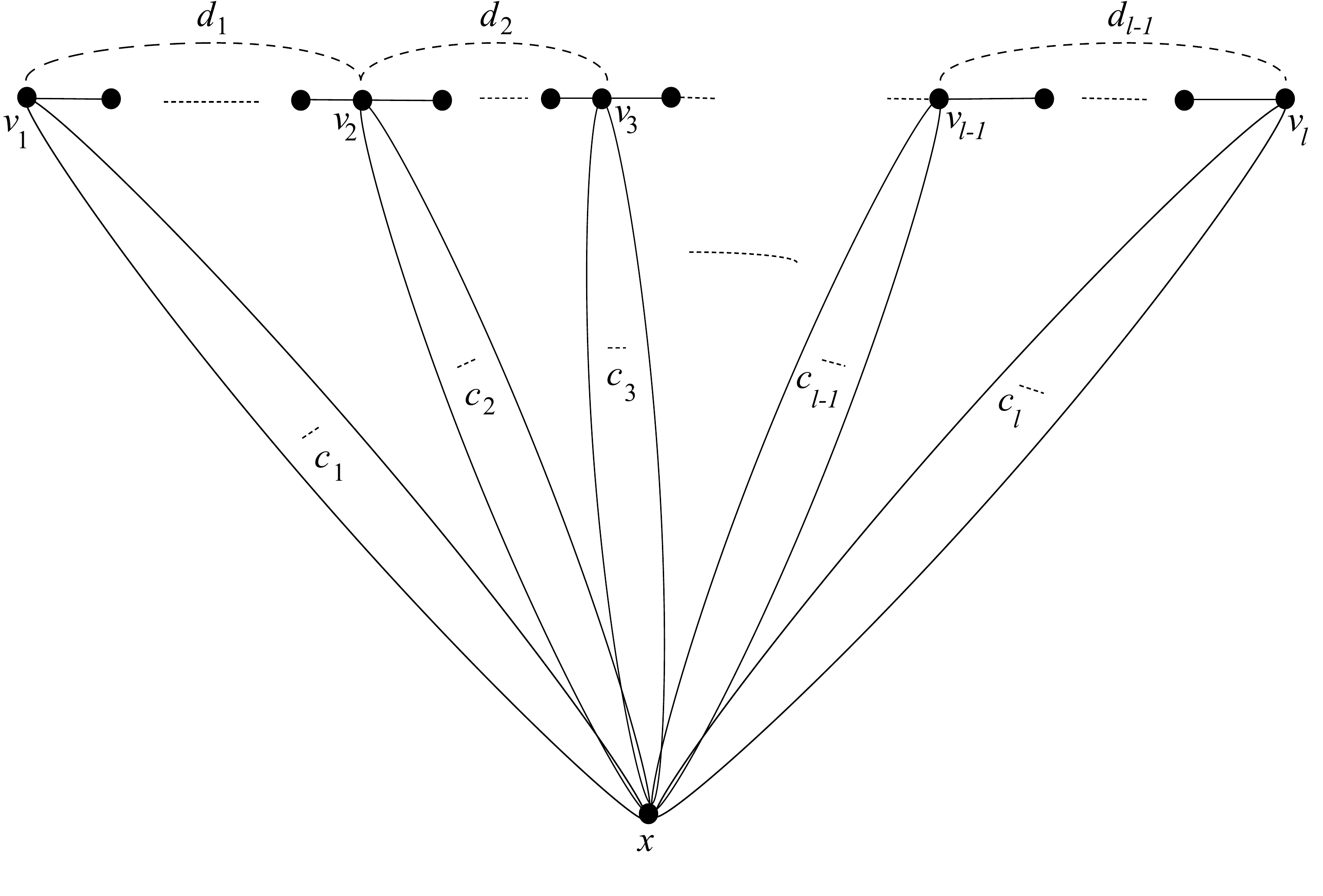}
 \caption{A multi-fan $F(c,d)$.}
 \label{fig:fanex}
\end{figure} 

Note that the ordinary {\em fan} coincides with the multi-fan with parameters $c=(1,1,\ldots,1)$ and $d=(1,1,\ldots,1)$. Also, a multi-fan is a series parallel graph created by alternately adding parallel edges from $x$ to the $v_i$'s and adding series edges from each $v_i$ to $v_{i+1}$. 


We are now ready to prove our correspondence between snakes and multi-fans. 
\begin{thm}\label{thm:charsnake} Let $a_1,\dots ,a_n$ be integers with $a_1\ge 1$ and $a_i\ge 2$ for each $i=2,\dots ,n$. The snake $S(a_1,a_2,\ldots,a_n)$ is isomorphic to the graphic matroid associated to the multi-fan $F(c,d)$, where
$$c=\begin{cases} (a_1+1) &\mbox{if } n=1, \\
(a_1,a_3-1,\ldots,a_{2k-1}-1,a_{2k+1}) &\mbox{if } n=2k+1>1, \\
(a_1,a_3-1,\ldots,a_{2k-1}-1,1) & \mbox{if } n=2k>1.\end{cases}$$ and 
 $$d=(a_2-1,a_4-1,\dots ,a_{2k}-1).$$
\end{thm}

\begin{proof}
  Given the intervals $A_1,\ldots A_{r(M)}$ representing $M=S(a_1,a_2,\ldots,a_n)$ as a transversal matroid, we associate a fan $F$ to $M$, such that the vertices $w_1,\ldots w_{r(M)}$ on the path of $F$ correspond to $A_1,\ldots A_{r(M)}$ in this order. Moreover, the number of parallel edges of a vertex $w_i$ to the special vertex $x$ of $F$ is the number of elements of $A_i$, that are not contained in any other $A_j$. 
  See Figure~\ref{fig:snakeandfan} for an illustration.
  Note that the parameters of $F$ depend on those of the snake $M$ exactly as claimed in the statement of the theorem. 
  Furthermore, note that by the definition of snakes, any two consecutive intervals $A_i, A_{i+1}$ share precisely one element. This gives that the edges incident to vertex $w_i$ of $F$ correspond to the elements of $M$ contained in $A_i$. 
  

%
%
%
%

 \begin{figure}[ht] 
\centering
 \includegraphics[width=.7\textwidth]{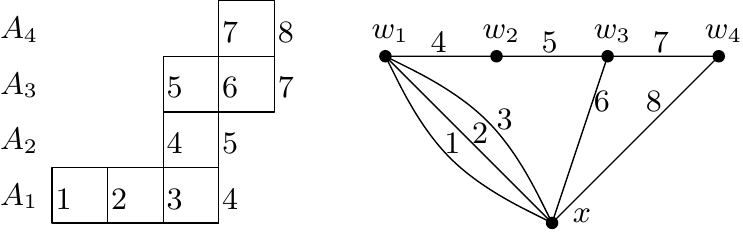}
 \caption{A snake and its associated multi-fan.}
 \label{fig:snakeandfan}
\end{figure} 

To show that both matroids are isomorphic, we give a correspondence between their bases. That is, we prove that the spanning trees of $F$ are in bijection with the transversals of $A_1, \ldots, A_{r(M)}$. Let $T$ be a spanning tree of $G$. We root $T$ at $x$ and orient all its edges away from $x$. Now, every edge $e$ of $T$ corresponds to an element of the ground set of $M$. We associate $e$ to the vertex $w_i$ it is oriented to, which in turn corresponds to an $A_i$ containing $e$. Since every vertex except $x$ has indegree $1$, this mapping proves that the edges of $T$ form a transversal of $A_1, \ldots, A_{r(M)}$. 

Conversely, suppose that we are given a transversal $T=\{x_1, \ldots x_{r(M)}\}$ of $A_1,\dots ,A_{r(M)}$. Since both the sequence of left endpoints and the sequence of right endpoints of $A_1,\dots ,A_{r(M)}$ are increasing, , we may assume $x_1\leq\ldots\leq x_{r(M)}$ and that $x_i \in A_i$ for all $i$. An edge corresponding to an element $x_i$ of the transversal can be oriented towards the vertex $w_i$ representing the $A_i$ to which $x_i$ is assigned. Since $T$ is a transversal, every vertex on the path of $F$ has indegree one. This implies, that $x$ has outdegree at least $1$. Moreover, vertex $x$ has indegree $0$. Consequently, the obtained graph is spanning and contains no cycles. We have obtained a tree rooted in $x$ and oriented away from $x$.
This gives the desired bijection.
%
%
\end{proof}

We remark that an alternative way of proving that snakes correspond to multi-fans is to use~\cite[Theorem 6.7]{Bon-06} where the structure of LPMs is described in terms of principal extensions.
The following corollary will be useful later on, when treating matroid duality with respect to snakes. 


\begin{cor}\label{cor:dualsnake} Let 
$a_1,\dots, a_n$ be integers with $a_1\ge 1$ and $a_i\ge 2$ for each $i=2,\dots ,n$. Then, the dual matroid $S^*(a_1,\dots ,a_{n})$ is isomorphic to the graphic matroid associated to the multi-fan $F(c',d')$ with
$$c'=\begin{cases} (1,a_2-1,a_4-1,\dots ,a_{2k}-1,1) &\mbox{if } n=2k+1\mbox{ and } a_1>1, \\(1,a_2-1,a_4-1,\dots ,a_{2k-2}-1,a_{2k}) & \mbox{if } n=2k\mbox{ and } a_1>1,\\
(a_2+1) & \mbox{if } n=2\mbox{ and } a_1=1,\\
(a_2,a_4-1,\dots ,a_{2k}-1,1) &\mbox{if } n=2k+1\mbox{ and } a_1=1,
\\(a_2,a_4-1,\dots ,a_{2k-2}-1,a_{2k}) & \mbox{if } n=2k\mbox{ and } a_1=1,
\end{cases}$$ 

and 

$$d'=\begin{cases} (a_1-1,a_3-1,\dots ,a_{2k+1}-1) &\mbox{if } n=2k+1\mbox{ and } a_1>1, \\(a_1-1,a_3-1,\dots ,a_{2k-1}-1) & \mbox{if } n=2k\mbox{ and } a_1>1,\\
(a_3-1,a_5-1,\dots ,a_{2k+1}-1) &\mbox{if } n=2k+1\mbox{ and } a_1=1,
\\(a_3-1,a_5-1,\dots ,a_{2k-1}-1) & \mbox{if } n=2k\mbox{ and } a_1=1.
\end{cases}$$ 
\end{cor}
\begin{proof}
Observation~\ref{obs:snakedual} yields a snake representation of $S^*(a_1,\dots ,a_{n})$. Theorem~\ref{thm:charsnake} applied to this snake yields the multi-fan as claimed above.
%
%
%
%
\end{proof}
 

The following lemma provides a formula for the number of acyclic orientations of a multi-fan.

\begin{lem}\label{lem:acyclic} Let  $c =(c_1,\ldots, c_{\ell})$ and $ d=(d_1,\ldots,d_{\ell-1})$ be vectors of positive integers. Then
 $$\alpha(F(c,d))= 2\prod_{j=1}^{\ell-1}\left(2^{d_j+1}-1\right).$$
\end{lem}

\begin{proof}
First observe that parallel edges must be oriented in the same direction. Start by orienting the first bundle of $c_1$ parallel edges in one of the two possible ways. Then the $d_1$ edges between $v_1$ and $v_2$ together with the second bundle of $c_2$ parallel edges can be oriented in a total of $2^{d_1+1}$ ways, all but exactly one of which are acyclic. Then the $d_2$ edges between $v_2$ and $v_3$
plus the third bundle of $c_3$ parallel edges can be oriented in $2^{d_2+1}-1$ acyclic ways, and so on, obtaining
$$2\prod_{j=1}^{\ell-1}\left(2^{d_j+1}-1\right)$$
acyclic orientations, as desired.  
\end{proof}


%

%

\begin{prop}
  \label{prop:prodsnakes}
  For any positive integer $n$ and $a_1,\dots ,a_n$ integers with $a_1\ge 1$ and $a_i\ge 2$ for each $i=2,\dots,n$, we have
  
  \begin{align}
    \label{eq:snakeprod}
    T(S(a_1,\ldots,a_n);0,2)\cdot T(S(a_1,\ldots,a_n);2,0))=2^2 \prod_{i=1}^n (2^{a_i}-1).  
  \end{align}
  
\end{prop}

\begin{proof} The case of the trivial snake $S(1)$ is trivial. Now, considering the snake or its dual we can suppose that $a_1\ge 2$ by Observation~\ref{obs:snakedual}. By Theorem~\ref{thm:charsnake},
we can express $T(S(a_1,\ldots,a_n);2,0)$ as $\alpha(F(c,d))$. Note that the vector $d$ does not depend on the parity of $n$ and is the only one taken into account by Lemma~\ref{lem:acyclic}. We can thus conclude:

$$\begin{array}{ll}
T(S(a_1,\ldots,a_{n});2,0)& = \alpha(F(c,d))\\
& = 2\prod\limits_{\scriptsize \begin{array}{c}i=2\\ i-\text{even}\end{array}}^{n} (2^{a_i-1+1}-1)\\
& =2\prod\limits_{\scriptsize \begin{array}{c}i=2\\ i-\text{even}\end{array}}^{n} (2^{a_i}-1).
\end{array}$$

Now, we use $T(S(a_1,\ldots,a_n);0,2)=T(S^*(a_1,\ldots,a_{n});2,0)$, which we can express as $\alpha(F(c',d'))$ by Corollary~\ref{cor:dualsnake}. Furthermore since $a_1\geq 2$ the vector $d'$ in the corresponding multi-fan $F(c',d')$ does not depend on the parity of $n$, i.e., in both cases comprises all $a_i$ with odd index, and is the only part of the parameters of $F(c',d')$ that is taken into account by Lemma~\ref{lem:acyclic}. We conclude:
$$\begin{array}{ll}
   T(S(a_1,\ldots,a_{n});2,0)& = T(S^*(a_1,\ldots,a_{n});0,2)\\
& = \alpha(F(c',d'))\\
& = 2\prod\limits_{\scriptsize \begin{array}{c}i=1\\ i-\text{odd}\end{array}}^{n} (2^{a_i-1+1}-1)\\
& =2\prod\limits_{\scriptsize \begin{array}{c}i=1\\ i-\text{odd}\end{array}}^{n} (2^{a_i}-1).
\end{array}$$
Obtaining,
$$T(S(a_1,\ldots,a_n);0,2)\cdot T(S(a_1,\ldots,a_n);2,0))=2^2 \prod_{i=1}^n (2^{a_i}-1).$$
\end{proof}

Now we turn our attention to $T(S;1,1)$. We will count the number of bases of a snake directly from its diagram. Let $\mathrm{Fib}(n)$ be the set of all binary sequences $b=(b_1,\ldots,b_n)$ of length $n$ such that there are no two adjacent $1$'s.

\begin{prop}
\label{prop:bsnakes}
  For any positive integer $n$ and $a_1,\dots ,a_n$ integers with $a_1\ge 1$ and $a_i\ge 2$ for each $i=2,\dots,n$ we have
  
  \begin{align}
   \label{eq:snakebases}
    T(S(a_1,\ldots,a_n),1,1)=\sum_{b \in \mathrm{Fib}(n+1)} \prod_{i=1}^n (a_i-1)^{1-|b_{i+1}-b_i|}.
  \end{align}

  Furthermore, for $n>1$ the following recursion holds:
  
  \begin{align}
   \label{eq:snakerec}
     \begin{split}
       T(S(a_1,\ldots,a_n),1,1)=&T(S(a_1,\ldots,a_{n-1}),1,1)+\\
       &(a_n-1)\cdot T(S(a_1,\ldots,a_{n-1}-1),1,1),
     \end{split}    
   \end{align}
  where we set $S(a_1,\ldots,a_{n-1}-1):=S(a_1,\ldots,a_{n-2})$ if $a_{n-1}=2$ and $n>2$.
\end{prop}

\begin{proof}

  Consider the snake $S(a_1,\ldots,a_n)$. If $n=1$ and $a_1=1$, on both sides of the equation we have $2$. In any other case, the snake has at least two squares. By duality, we may suppose that the snake starts with two adjacent horizontal squares, that is $a_1\ge 2$.
	
	Let $S(a_1,\ldots,a_n)=M[P,Q]$. We will label some points on the paths $P$ and $Q$ with $0$'s and $1$'s. As we explain the labeling, Figure~\ref{fig:onezero} may be used as a reference for the case $n=4$. We label as follows. On the snake consider $C_1$ the first square, $C_{n+1}$ the last square and for each $i\in \{2,\ldots,n\}$ let $C_i$ be the $(i-1)$-th square in which the snakes changes direction. For each square $C_i$ let $u_i$ be its upper left vertex and $v_i$ its lower right vertex. We label each $u_i$ with $1$ if $i$ is odd and with $0$ if $i$ is even. We label each $v_i$ with the label opposite to the one in $u_i$.
	
\begin{figure}[ht] 
  \centering
  \includegraphics[width=.5\textwidth]{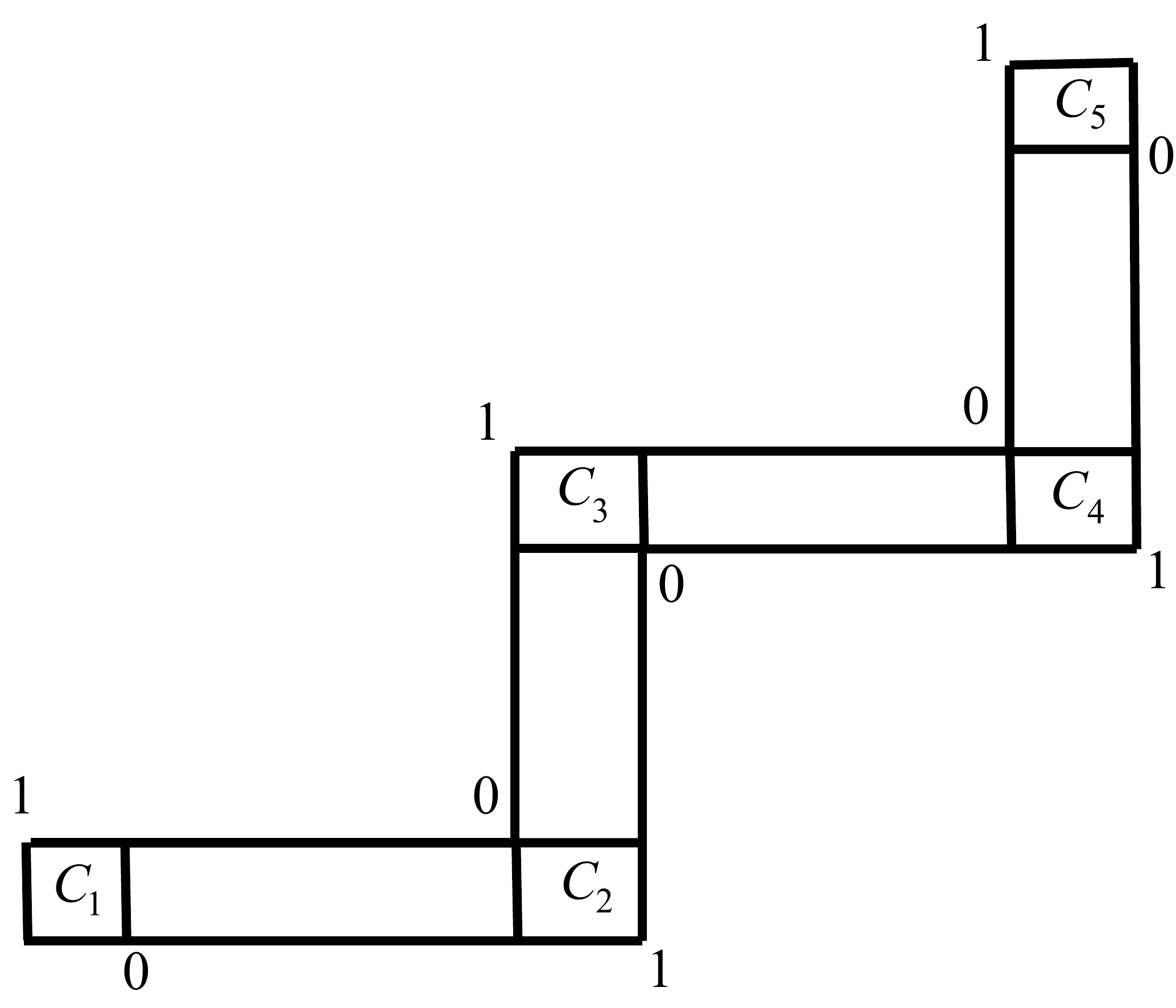}
  \caption{Labeling of $S(a_1,a_2,a_3,a_4)$ with zeros and ones.}
  \label{fig:onezero}
\end{figure} 
	
	Consider a lattice path within the diagram of $S(a_1,\ldots,a_n)$. For every $1\leq i\leq n+1$ this lattice path has to go through exactly one of the vertices $u_i$, $v_i$. Therefore, for each lattice path we can assign a binary sequence of length $n+1$. We claim that the formula in Equation \eqref{eq:snakebases} counts the number of lattice paths according to their corresponding binary sequence.
	
	First, it is impossible to go consecutively from a vertex labeled $1$ to another vertex labeled $1$. Therefore all the possible binary sequences are in $\mathrm{Fib}(n+1)$. Now we take a binary sequence $B=(b_1,\ldots,b_{n+1})$ and we count to how many lattice paths it corresponds. Consider the segment of the path that goes from the vertices in square $C_i$ to the vertices in square $C_{i+1}$.
	
	\begin{itemize}
		\item If we go from the vertex with label $0$ to the vertex with label $1$ or vice versa, there is exactly one way in which we can do it.
		\item There are exactly $a_i-1$ ways to go from the vertex with label $0$ to the vertex with label $0$, corresponding to the choice of the unique $N$-step, respectively $E$-step.
	\end{itemize}
	
	  Thus if the binary sequence is $B$, we can go from the vertices in $C_i$ to the vertices in $C_{i+1}$ in $(a_i-1)^{1-|b_{i+1}-b_i|}$ ways, and therefore there are
	
	\[
		\prod_{i=1}^n (a_i-1)^{1-|b_{i+1}-b_i|}
	\]
	
	lattice paths with corresponding sequence equal to $B$. This shows that the formula is correct.
	
	The recursive formula can be proved using Equation (\ref{eq:snakebases}), but we provide a combinatorial proof. To do so we verify whether the lattice path has gone through the upper right vertex of $C_n$ or not. If it did, by definition there are $T(S(a_1,\ldots,a_{n-1});1,1)$ ways of getting to that vertex and then the path to the end is completely determined. If it did not, then in square $C_n$ the path has to go through the vertex with label $0$, which can be done in $T(S(a_1,\ldots,a_{n-1}-1),1,1)$ ways. This has to be multiplied by the $a_n-1$ ways to complete the path avoiding the upper right vertex of $C_n$. This completes the argument.

\end{proof}

Notice that when $a_1=a_2=\ldots=a_n=2$ we are summing only $1$'s over all the sequences of $\mathrm{Fib}(n+1)$. It is a folklore result that the number of such sequences is the $(n+3)$-rd Fibonacci number, and thus Proposition~\ref{prop:bsnakes} can be regarded as a lattice path generalization of this. Indeed, the fact that the number of spanning trees of ordinary fans is counted by the Fibonacci numbers has been observed several times, see~\cite{Ban-81,Fie-74,Hil-74}.

%
%


\section{The multiplicative Merino-Welsh conjecture for LPMs}
\label{sec:mwconj}

We will now prove that the strongest version of Conjecture~\ref{conj:MMW} is true for LPMs. Notice that equality may hold. An easy example is the trivial snake. Since the Tutte polynomial of a direct sum is the product of the polynomials of the components of the direct sum, a direct sum of trivial snakes also yields equality.

More specifically, in this section we prove Theorem~\ref{thm:MWLPM} which is an improvement on the desired inequality by a constant factor except in the trivial cases mentioned above. 

We provide an inductive proof. The strategy is as follows:

\begin{itemize}
    \item We prove the theorem for snakes.
    \item We show that any connected LPM $M$ either is a snake, or it has an element $e$ such that both $M\setminus e$ and $M/e$ are connected LPMs with fewer elements.
    \item We state a straightforward lemma for proving the inequality for $M$ from the veracity of the inequality for $M\setminus e$ and $M/e$.
	\item We extend the result to disconnected but LC LPM.
\end{itemize}

Before starting with the first step in the strategy, let us make a remark. In Section~\ref{sec:lpmsnakes} we have shown that snakes are series parallel graphic matroids. Therefore, Conjecture~\ref{conj:MMW}.3 can be proved for snakes using the result in~\cite{Nob-14}. However, for the whole strategy to work we will need to prove first the sharper inequality for snakes. Thus we will need the precise results on the Tutte polynomial provided by Proposition~\ref{prop:prodsnakes} and Proposition~\ref{prop:bsnakes}.

\begin{prop}
    \label{prop:WMsnakes}
    If $M$ is a non-trivial snake, then    
    \[
      T(M;2,0)\cdot T(M;0,2)\geq \frac{4}{3}\cdot T(M;1,1)^2.
    \]  
\end{prop}

\begin{proof}
  Let $M=S(a_1,\ldots,a_n)$ be a non-trivial snake. We proceed by induction on $n$. If $n=1$, then $M=S(a)$ and since the snake is non-trivial we have $a\geq 2$. Now, $T(S(a);1,1)$ is the number of lattice paths in its diagram which is clearly $a+1$. By Equation~\eqref{eq:snakeprod}, we have to prove that
  
  \[
    4 \cdot (2^a-1)\geq \frac{4}{3}\cdot (a+1)^2.
  \]
  
  Since $a\geq 2$, we have $a^2\geq a+2$. Using the binomial formula we get
  
  \begin{align*}
    4\cdot ((1+1)^a-1)&\geq 4\cdot \left(1+a+\frac{a(a-1)}{2}-1\right)\\
     &=2a^2+2a= \frac{4}{3} \cdot a^2 + \frac{2}{3} \cdot a^2 +2a \geq \frac{4}{3} \cdot a^2 + \frac{2}{3} \cdot (a+2) +2a\\
		 &= \frac{4}{3} \cdot (a^2+2a+1)=\frac{4}{3} \cdot (a+1)^2.
  \end{align*}
  
	We need another induction base: the snakes $S(2,a)$. By using Equations \eqref{eq:snakeprod} and \eqref{eq:snakebases}, we need to prove that
	
	\[
		4 \cdot 3 \cdot (2^a-1)\geq \frac{4}{3} \cdot (2a+1)^2.
	\]
	
	Recall that $a\geq 2$. By using the binomial formula again we have
	
	\begin{align*}
	   4\cdot 3 \cdot (2^a-1)&\geq 12 \cdot \left(1+a+\frac{a(a-1)}{2}-1\right)\\
													&=6a^2+6a=\frac{4}{3} \cdot (4a^2+4a) + \frac{2}{3} (a^2 + a) \geq \frac{4}{3} \cdot (4a^2+4a+1)\\
													&=\frac{4}{3} \cdot (2a+1)^2.
	\end{align*}
	
  This proves our induction bases. We now suppose that the conclusion is true for $1,2,\ldots,n-1$ and we consider the snake $S(a_1,\ldots, a_{n-1},b)$ with $n,b\geq 2$ and if $n=2$, then $a_1\geq 3$. By using Equation \eqref{eq:snakerec}, we have that:
  
  \begin{align*}  
  T(S(a_1,\ldots,b),1,1)& =T(S(a_1,\ldots,a_{n-1}),1,1) \ +\\
       & \ \ \ \ (b-1)\cdot T(S(a_1,\ldots,a_{n-1}-1),1,1).
  \end{align*}
We may now use the induction hypothesis. Notice that $a_{n-1}-1$ may become $1$, but only if $n>2$. In this case we consider $S(a_1,\ldots,a_{n-2})$. Thus, we can always conclude that $T(S(a_1,\ldots,b),1,1)$ is less than or equal to

  \[
      \frac{\sqrt{3}}{2}\cdot 2\cdot \prod_{i=1}^{n-1} (2^{a_i}-1)^{1/2}+\frac{\sqrt{3}}{2}\cdot (b-1) \cdot 2\cdot (2^{a_{n-1}-1}-1)^{1/2}\cdot \prod_{i=1}^{n-2} (2^{a_i}-1)^{1/2}
  \]
  
  which can be factorized as
  
  \[
    \frac{\sqrt{3}}{2}\cdot 2\cdot \left( \prod_{i=1}^{n-2} (2^{a_i}-1)^{1/2}\right)\cdot \left((2^{a_{n-1}}-1)^{1/2}+(b-1)\cdot (2^{a_{n-1}-1}-1)^{1/2}\right).
  \]
  
  Therefore, to get the two extra factors that we need it will be enough to prove that for any $a_{n-1}\geq 2$ and $b\geq 2$ we have
  
  \[
    (2^{a_{n-1}}-1)^{1/2}+(b-1)\cdot (2^{a_{n-1}-1}-1)^{1/2} \leq (2^{a_{n-1}}-1)^{1/2}\cdot (2^{b}-1)^{1/2}.
  \]
	
	Dividing both sides by $(2^{a_{n-1}}-1)^{1/2}$ this becomes
	
	\[
		1+\frac{b-1}{\sqrt{2}}\cdot \left(1-\frac{1}{2^{a_{n-1}}-1}\right)^{1/2}\leq (2^{b}-1)^{1/2}.
	\]
	
	We will prove that for $b\geq 2$ the following stronger inequality holds
	
	\[
	  1+\frac{b-1}{\sqrt{2}}\leq (2^{b}-1)^{1/2}.
	\]
	
	By the binomial formula, $2^b\geq 1+b+\frac{b(b-1)}{2}$. Therefore,
	
	\[
		2^b-1 \geq \frac{b^2+b}{2}\geq \frac{b^2}{2}+\left(\sqrt{2}-1\right)b+ \frac{3}{2}-\sqrt{2}=\left(1+\frac{b-1}{\sqrt{2}}\right)^2.
	\]
	
	This proves the desired inequality and thus the proposition follows by induction.
	
\end{proof}

\begin{prop}\label{prop:LPdown} Let $M$ be a connected LPM. Then either $M$ is a snake or $M$ has an element $e$ such that both $M\setminus e$ and $M/e$ are connected LPMs different from the trivial snake.
\end{prop}

\begin{proof} Suppose that $M=M[P,Q]$ is a connected LPM that is not a snake. Let us consider the interior lattice point of $M$ that is highest and rightmost, say $p=(x,y)$. We claim that $e=x+y+1$ is the desired element of $M$, see Figure~\ref{fig:intpoint}. 
Indeed,~\cite[Corollary 2.17]{Bon-10} states that for any element $e$ of a connected LPM that is not the first or the last, the contraction $M/e$ is connected if and only if $e$ is in at least two sets in the presentation as a transversal matroid. Since this is the case for the above $e$, $M/e$ is connected. The connectivity of $M\setminus e$ follows by duality. 

%
%
%

\begin{figure}[ht] 
  \centering
  \includegraphics[width=.7\textwidth]{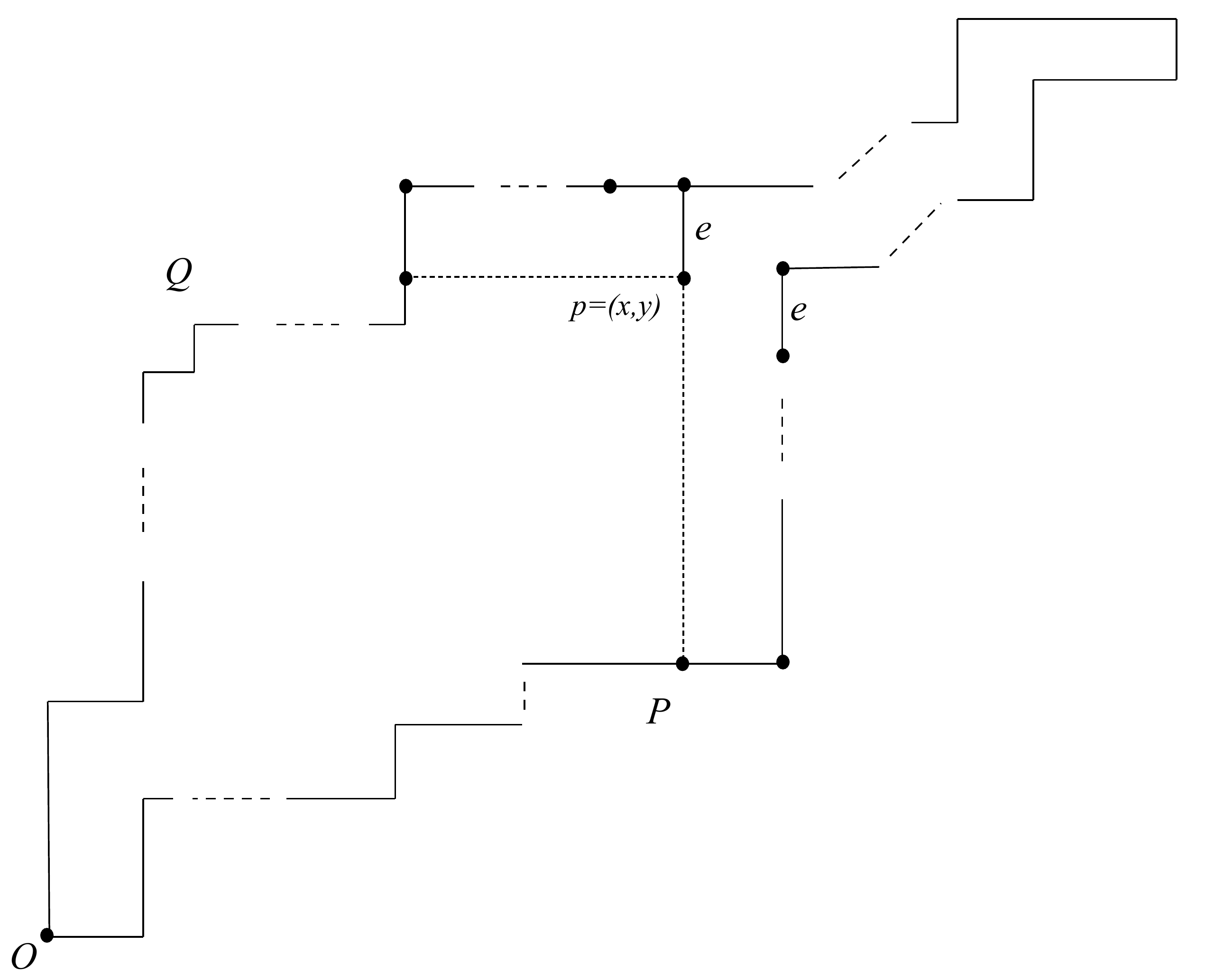}
  \caption{An LPM with an interior point $p$.}
  \label{fig:intpoint}
\end{figure}

\end{proof}

The following result is valid for general matroids. A version without the $\frac{4}{3}$ factor  is given in~\cite{Jackson2010}; see also~\cite[Lemma 2.2]{Nob-14}. The following proof is slightly different, and we include it for completeness.

\begin{lem}
  \label{lem:delcon}
	Let $M$ be a loopless and coloopless matroid and let $e$ be an element of its ground set. Suppose that the inequality in Theorem~\ref{thm:MWLPM} holds for $M\setminus e$ and for $M/e$. Then, the inequality also holds for $M$.
\end{lem}

\begin{proof}
  We define $p,q,r,s,t,u$ as follows:
	
	\begin{align*}
	  p&=T(M\setminus e;2,0),\; q=T(M\setminus e;0,2),\; r=T(M\setminus e;1,1),\\
		s&=T(M/e;2,0),\; t=T(M/e;0,2),\; u=T(M/e;1,1).
	\end{align*}
	
	Since $M$ is loopless and coloopless, we have that $T(M;x,y)=T(M\setminus e;x,y)+T(M/e;x,y)$. Therefore, we have to prove that
	
	\[
		(p+s)(q+t)\geq \frac{4}{3} \cdot  	(r+u)^2.
	\]
	
	By hypothesis, we know that $p\cdot q \geq \frac{4}{3} \cdot r^2$ and that $s\cdot t \geq \frac{4}{3} \cdot u^2$. Combining this and the Cauchy-Schwartz inequality we conclude as follows:
	
	\[
		(p+s)(q+t)\geq \left(\sqrt{pq}+\sqrt{st}\right)^2\geq \frac{4}{3} \cdot (r+u)^2.
	\]
	
\end{proof}
Notice that there is nothing special about $4/3$ in the lemma above in the sense that if $p\cdot q \geq k \cdot r^2$ and $s\cdot t \geq k\cdot u^2$ then $(p+s)(q+t)\ge k(r+u)^2$. The value $k=4/3$ is the one that gives equality for the snake $S(2)$.
\smallskip

We are now ready to prove our main result.

\begin{proof}[Proof of Theorem~\ref{thm:MWLPM}]

  First we prove the theorem for connected LPMs. In this proof we will only refer to LPMs different from the trivial snake. We proceed by induction on the number of elements. If the matroid has three elements, then a connected LPM with 3 elements is either $S(2)$ or its dual, for which we know that the theorem is true.
\smallskip
	
	Now suppose that the theorem is true for connected LPMs of less than $n$ elements. Let $M$ be a connected LPM with $n$ elements. If $M$ is a snake, then by Proposition~\ref{prop:WMsnakes} the inequality holds. Otherwise, by Proposition~\ref{prop:LPdown} we can find an element $e$ such that both $M\setminus e$ and $M/e$ are connected LPMs. Each of these has fewer elements than $M$, and thus by the inductive hypothesis the inequality holds for both of them. Therefore using Lemma~\ref{lem:delcon} we conclude that the inequality also holds for $M$. This completes the proof for connected LPMs.
	
	We are left with the case in which $M$ is LC but not connected. In this case we express $M$ as direct sum of connected LPMs $M_1$, $M_2$, $\ldots$, $M_n$. By our assumption at least one of them, say $M_1$, is not the trivial snake. For each $1\leq i\leq n$ let
	
	\[
		p_i=T(M_i;2,0),\; q_i=T(M_i;0,2),\; r_i=T(M_i;1,1).
	\]
	
We know that $p_1\cdot q_1\geq \frac{4}{3} \cdot r_1^2$ and that for each $2\leq i\leq n$ we have $p_i\cdot q_i\geq r_i ^2$. Using the fact that the Tutte polynomial of a direct sum is the product of the Tutte polynomials of the components we get:
	
	\begin{align*}
		T(M;2,0)\cdot T(M;0,2) &=\prod_{i=1}^n p_i \cdot \prod_{i=1}^n q_i = \prod_{i=1}^n (p_i\cdot q_i)\\
				      &\geq \frac{4}{3}\cdot \prod_{i=1}^n r_i^2 = \frac{4}{3}\cdot\left(\prod_{i=1} r_i\right)^2 = \frac{4}{3}\cdot T(M;1,1)^2.
	\end{align*}

  Therefore, the inequality is true for every LC LPM that is not a direct sum of trivial snakes.
	
\end{proof}

Theorem~\ref{thm:MWLPM} immediately yields the following corollary which confirms the multiplicative Merino-Welsh conjecture for LPMs.

\begin{cor}\label{cor:MWLPM}

Let $M$ be a lattice path matroid with no loops and no coloops. Then we have
      $$T(M;2,0)\cdot T(M;0,2)\geq T(M;1,1)^2$$	
and equality holds if and only if $M$ is a direct sum of trivial snakes.

\end{cor}

%





\bibliographystyle{plain}
\def\cprime{$'$}


\def\cprime{$'$}
\begin{thebibliography}{10}

\bibitem{Ban-81}
David~W. {Bange}, Anthony~E. {Barkauskas}, and Peter~J. {Slater}.
\newblock {Fibonacci numbers in tree counts for maximal outerplane and related
  graphs.}
\newblock {\em {Fibonacci Q.}}, 19:28--34, 1981.

\bibitem{Bon-10}
Joseph~E. Bonin.
\newblock Lattice path matroids: the excluded minors.
\newblock {\em J. Combin. Theory Ser. B}, 100(6):585--599, 2010.

\bibitem{Bon-06}
Joseph~E. Bonin and Anna de~Mier.
\newblock Lattice path matroids: structural properties.
\newblock {\em European J. Combin.}, 27(5):701--738, 2006.

\bibitem{Bon-03}
Joseph~E. Bonin, Anna de~Mier, and Marc Noy.
\newblock Lattice path matroids: enumerative aspects and {T}utte polynomials.
\newblock {\em J. Combin. Theory Ser. A}, 104(1):63--94, 2003.

\bibitem{Bon-07}
Joseph~E. Bonin and Omer Gim{\'e}nez.
\newblock Multi-path matroids.
\newblock {\em Combin. Probab. Comput.}, 16(2):193--217, 2007.

\bibitem{Cha-11}
Vanessa Chatelain and Jorge~Luis Ram{\'{\i}}rez~Alfons{\'{\i}}n.
\newblock Matroid base polytope decomposition.
\newblock {\em Adv. in Appl. Math.}, 47(1):158--172, 2011.

\bibitem{ChavezL2011}
Laura~E. {Ch\'avez-Lomel{\'\i}}, Criel {Merino}, Steven~D. {Noble}, and
  Marcelino {Ram{\'\i}rez-Ib\'a\~nez}.
\newblock {Some inequalities for the {T}utte polynomial.}
\newblock {\em {Eur. J. Comb.}}, 32(3):422--433, 2011.

\bibitem{Conde2009}
R.~Conde and C.~Merino.
\newblock Comparing the number of acyclic and totally cyclic orientations with
  that of spanning trees of a graph.
\newblock {\em Int. J. Math. Com.}, 2:79--89, 2009.

\bibitem{Del-12}
Emanuele {Delucchi} and Martin {Dlugosch}.
\newblock {Bergman Complexes of Lattice Path Matroids}.
\newblock {\em SIAM, J. Disc. Math}, 24(9):1916--1930, 2015.

\bibitem{Edmonds-65}
J.~{Edmonds} and D.R. {Fulkerson}.
\newblock Transversal and matroids partition.
\newblock {\em J. Res. Nat. Bur. Standards Sect. B}, 69 B:147--153, 1965.

\bibitem{Fie-74}
Daniel~C. {Fielder}.
\newblock {Fibonacci numbers in tree counts for sector and related graphs.}
\newblock {\em {Fibonacci Q.}}, 12:355--359, 1974.

\bibitem{Hil-74}
A.J.W. {Hilton}.
\newblock {Spanning trees and Fibonacci and Lucas numbers.}
\newblock {\em {Fibonacci Q.}}, 12:259--262, 1974.

\bibitem{Jackson2010}
Bill {Jackson}.
\newblock {An inequality for Tutte polynomials.}
\newblock {\em {Combinatorica}}, 30(1):69--81, 2010.

\bibitem{Merino1999}
C.~{Merino} and D.J.A. {Welsh}.
\newblock {Forests, colorings and acyclic orientations of the square lattice.}
\newblock {\em {Ann. Comb.}}, 3(2-4):417--429, 1999.

\bibitem{Mor-13}
Jason {Morton} and Jacob {Turner}.
\newblock {Computing the Tutte polynomial of lattice path matroids using
  determinantal circuits.}
\newblock {\em {Theor. Comput. Sci.}}, 598:150--156, 2015.

\bibitem{Nob-14}
Steven~D. {Noble} and Gordon~F. {Royle}.
\newblock {The Merino-Welsh conjecture holds for series-parallel graphs.}
\newblock {\em {Eur. J. Comb.}}, 38:24--35, 2014.

\bibitem{Sch-10}
Jay Schweig.
\newblock On the {$h$}-vector of a lattice path matroid.
\newblock {\em Electron. J. Combin.}, 17(1):Note 3, 6, 2010.

\bibitem{Sch-11}
Jay Schweig.
\newblock Toric ideals of lattice path matroids and polymatroids.
\newblock {\em J. Pure Appl. Algebra}, 215(11):2660--2665, 2011.

\bibitem{Thomassen2010}
Carsten {Thomassen}.
\newblock {Spanning trees and orientation of graphs.}
\newblock {\em {J. Comb.}}, 1(2):101--111, 2010.

\bibitem{Welsh1976}
D.J.A. {Welsh}.
\newblock {Matroid theory.}
\newblock {L.M.S. Monographs. Vol. 8. London - New York - San Francisco:
  Academic Press, a subsidiary of Harcourt Brace Jovanovich, Publishers. XI,
  433 p. (1976).}, 1976.

\bibitem{Welsh1999}
Dominic {Welsh}.
\newblock {The Tutte polynomial.}
\newblock {\em {Random Struct. Algorithms}}, 15(3-4):210--228, 1999.

\end{thebibliography}

\end{document}